\documentclass[a4paper,11pt]{amsart}

\usepackage{amsfonts,latexsym,rawfonts,amsmath,amssymb,amsthm, mathrsfs, lscape}
\usepackage{graphicx}
\usepackage[all]{xy}
\usepackage[english]{babel}
\usepackage[utf8]{inputenc}
\usepackage{xfrac}

\usepackage{cite}

\usepackage{array, tabularx}

\usepackage{bm}
\bmdefine{\bX}{X}
\bmdefine{\ba}{$\alpha$}

\usepackage{setspace}
\usepackage{textcomp}

\usepackage{color}

\usepackage{xparse}
\usepackage{soul}

\usepackage{graphicx}

\newcommand{\referenza}{}

\newtheorem{thm}{Theorem}[section]
\newtheorem*{thm*}{Theorem \referenza}
\newtheorem*{thm**}{Theorem}

\newtheorem*{cor*}{Corollary \referenza}
\newtheorem{lem}[thm]{Lemma}
\newtheorem*{lem*}{Lemma \referenza}

\newtheorem*{clm*}{Claim \referenza}
\newtheorem{prop}[thm]{Proposition}
\newtheorem*{prop*}{Proposition \referenza}

\newtheorem*{ass*}{Assumption \referenza}

\newtheorem{conj}[thm]{Conjecture}
\newtheorem*{conj*}{Conjecture \referenza}

\newtheorem{rmk}[thm]{Remark}

\newtheorem{exa}[thm]{Example}
\newtheorem{defi}[thm]{Definition}

\theoremstyle{plain}

\def\ra{\rightarrow}
\def \N {\mathbb N}

\def \C {\mathbb C}
\def \Z {\mathbb Z}
\def \P {\mathbb P}

\allowdisplaybreaks

\usepackage[hypertexnames=false,
backref=page,
    pdfpagemode=UseNone,
    breaklinks=true,
    extension=pdf,
    colorlinks=true,
    linkcolor=blue,
    citecolor=blue,
    urlcolor=blue,
]{hyperref}

\makeatletter
\@namedef{subjclassname@2020}{
  \textup{2020} Mathematics Subject Classification}
\makeatother

\begin{document}

\title[lcK threefolds with $a(X)=2$]{On locally conformally K\"ahler threefolds with algebraic dimension two}

\author{Daniele Angella}
\address{Dipartimento di Matematica e Informatica ``Ulisse Dini''\\
	Universit\`a degli Studi di Firenze,
	viale Morgagni 67/a,
	50134 Firenze, Italy}
\email{daniele.angella@gmail.com, daniele.angella@unifi.it}

\author{Maurizio Parton}
\address{	Dipartimento di Economia\\
	Universit\`a di Chieti-Pescara,
	viale della Pineta 4,
	65129 Pescara, Italy}
\email{parton@unich.it}

\author{Victor Vuletescu}
\address{Faculty of Mathematics and Informatics\\
	University of Bucharest,
	Academiei st. 14,
	Bucharest, Romania}
\email{vuli@fmi.unibuc.ro}


\thanks{During the preparation of this work, the first-named author has been supported by project PRIN2017 ``Real and Complex Manifolds: Topology, Geometry and holomorphic dynamics'' (code 2017JZ2SW5), by project SIR2014 ``AnHyC Analytic aspects in complex and hypercomplex geometry" (code RBSI14DYEB), and by GNSAGA of INdAM. The second-named author has been supported by GNSAGA of INdAM.
The third-named author has been partially supported by UEFISCDI grant PCE 30/2021\\
The authors have been supported by the Research in Pairs program by CIRM-Fondazione Bruno Kessler.}

\begin{abstract}
The paper is part of an attempt of understanding non-K\"ahler threefolds. We start by looking at compact complex non-K\"ahler threefolds with algebraic dimension two and admitting locally conformally K\"ahler metrics. Under certain assumptions, we prove that they are blown-up quasi-bundles over a projective surface.
\end{abstract}

\maketitle

\section*{Motivation and outline of the paper}

The paper is part of an attempt of understanding non-Moishezon non-K\"ahler threefolds. We start by looking at the simplest case: threefolds $X$ with algebraic dimension $a(X)=2$. In this case, it is classically known that they are bimeromorphic to elliptic fibrations over projective surfaces, that is, there exists a smooth bimeromorphic threefold $X^*$ and a surjective holomorphic map $f\colon X^*\to B$ whose general fibres are smooth elliptic curves. In fact, $X^*$ is an {\em algebraic reduction} of $X$. The main goal of the paper is to give a description of $X^*$, and then retrieve information about $X$.

By generalizing what happens on non-K\"ahler surfaces $S$ of algebraic dimension $a(S)=1$ \cite[Proposition 3.17]{brinzanescu}, we prove that, under mild assumptions that we now describe, the algebraic reduction is a {\em quasi-bundle}, namely, the fibres with the reduced structure are smooth elliptic isomorphic curves.
The main idea of our strategy is  inspired by the Lefschetz hyperplane theorems in algebraic geometry. More exactly, like in Fuijta's paper \cite{fujita}, we consider divisors $H$ on $B$, and look at their preimages $S_H:=f^{-1}(H)\subset X$ that gives rise to elliptic surfaces $S_H\stackrel{f_{\vert S_H}}{\ra} H$. Then we take advantage of the known results for compact complex non-K\"ahler surfaces.

Unfortunately, the information about non-K\"ahlerianity passes badly from $X$ to $S_H$: we will exhibit in Example \ref{ex1} a non-K\"ahler $X$ having a smooth K\"ahler $S_H$. Still, if $X^*$ is lcK and $H$ is general enough, we are able to show that $S_H$ is non-K\"ahler as well. Here, by {\em locally conformally K\"ahler (lcK)}, we mean that there is a {\itshape non-K\"ahler} Hermitian metric that locally (but not globally) admits a conformal change to a K\"ahler metric.

Since the algebraic reduction is defined only up to bimeromorphisms, we are forced to deal with the  resolutions of singularities.
Unfortunately K\"ahlerianity in dimension at least $3$ behave badly with respect to bimeromorphisms: the blow-up of a K\"ahler manifold is K\"ahler, but the blow-down of a K\"ahler manifold may well be non-K\"ahler. 
With lcK metrics, the situation is even worse: it may happen that even a blow-up of an lcK manifold carries no lcK metric. So, we introduce the {\itshape ad hoc} notion of {\em weak locally conformally K\"ahler} (wlcK, for short) structure, see Section \ref{sec:wlck}. The main point is that the wlcK condition is behaving well under blow-up: if $X$ is wlcK and $\hat{X}$ is an arbitrary blow-up of it, then $\hat{X}$ has a wlcK structure too. 

Returning to bimeromorphisms, it is known that one can always reach the algebraic reduction $X^*$ starting from $X$ by means of a chain of blow-ups and blow-downs with smooth centres, the so-called Weak Factorization Theorem. Since wlcK structures behave badly with respect to blow-downs, we will work under the assumption that the Strong Factorization Conjecture holds true: that is, there are a (smooth) threefold $\hat{X}$ and maps $\sigma$ and $c$ compositions of blow-ups with smooth centres such that
$$
\xymatrix{
& \hat X \ar[ld]_{\sigma} \ar[rd]^{c} & \\
X \ar@{<-->}[rr]_{\psi} & & X^* .
}
$$
Thus, if $X$ has an lcK metric, then $\hat{X}$ still has a wlcK structure, so in this case we have:

\renewcommand{\referenza}{\ref{thm:smooth-fibres}}
\begin{thm*}
Let $X$ be a compact complex threefold with $a(X)=2$, endowed with an lcK structure $(\omega,\theta)$. Assume the Strong Factorization Conjecture, and consider the algebraic reduction
\begin{equation*}
\xymatrix{
& \hat X \ar[ld]_{\sigma} \ar[rd]^{c} \\
X \ar@{<-->}[rr]_{\psi} & & X^* \ar[d]	_{f} \\
& & B .
}
\end{equation*}
Then all the smooth one-dimensional fibres of $f$ are isomorphic.
\end{thm*}

In the above notation, it may still happen that the map $f$ is a non-flat morphism, {\itshape i.e.} some of its fibres may have dimension $2$. Applying the Hironaka flattening \cite{hironaka-AJM} to $f$ may result into putting heavy singularities on the flattening, so we will make the further assumption that $\psi$ can be factored as above with $f$ {\em flat morphism} (compare Examples \ref{ex1} and \ref{ex2}).
Moreover, we will also assume that there is such an $X^*$ which is {\em minimal} in the sense of Mori program, that is $K_{X^*}$ is nef.
Notice that the existence of minimal models for compact elliptic threefolds (that are not uniruled) has been solved by Grassi \cite{grassi} for projective threefolds, and by H\"oring and Peternell \cite{horing-peternell} more generally for K\"ahler threefolds: it is reasonable to hope that this holds good also in the non-K\"ahler case.
Finally, we will also assume that the singular locus $S(f)$ of the fibration $f$ is a simple normal crossing divisor.
In our case, we wonder whether the assumptions above should hold in general: but we were so far unable to give a rigourous proof.

Under these assumptions, we are able to prove:
\renewcommand{\referenza}{\ref{thm:quasibundle}}
\begin{thm*}
All the singular fibres of $f$ are in fact just multiple of the general fibre, that is, $X^*$ is a quasi-bundle.
\end{thm*}

Moreover, using the properties of non-K\"ahler quasi-bundles, and assuming also that $\hat{X}$ is lcK, we can prove:

\renewcommand{\referenza}{\ref{thm:structure}}
\begin{thm*}
There is a bimeromorphic morphism $\varphi:X\ra X^*$
which  is a proper modification of $X^*$ at points.
\end{thm*}

\bigskip 

{\itshape Acknowledgements.}
This work has been partly written in Trento thanks to a Research in Pairs at CIRM: the authors would like to warmly thank CIRM-Fondazione Bruno Kessler for their support and the very pleasant working environment.
Special thanks to Marian Aprodu for careful reading and interacting with us during the preparation of the paper. Many thanks also to Cezar Joi\c{t}a and Liviu Ornea for many useful discussions and suggestions.
We also warmly thank the anonymous Referees for their valuable comments.

\section{Prerequisites  and examples}
Let $X$ be a compact threefold of algebraic dimension $a(X)=2$. We known that $X$ fibres over a smooth projective surface up to bimeromorphisms \cite[page 25]{ueno}: more precisely, there is an {\em algebraic reduction}
\begin{equation}\label{eq:alg-red}
 \xymatrix{
X \ar@{<-->}[r]^\psi & X^* \ar[d]^f \\
& B
}
\end{equation}
where $X^*$ is a compact complex threefold bimeromorphic to $X$ through the bimeromorphism $\psi$, and $f$ is a surjective, proper, holomorphic map onto the compact projective surface $B$, with connected fibres \cite[Proposition 3.4]{ueno}.
We know that the general fibre of $f$ is smooth and elliptic \cite[Theorem 2]{kawai}, \cite[Theorem 12.4]{ueno}.

In the following, we describe the setting for the main results.

\subsection*{Flatness}
The following two examples show that $f$ being non-flat might result in unexpected behaviour for the degeneration of fibres.

\begin{exa}\label{ex1}
Let $H^3=\mathbb C^3 \setminus \{0\} / \mathbb Z$ be the Hopf threefold, where $\mathbb Z$ acts as $z \mapsto 2z$, and $H^3 \ni (z_1,z_2,z_3) \mapsto [z_1:z_2:z_3] \in \mathbb P^2$ be the Hopf fibration. Consider the blow-up of $\mathbb P^2$ at a point $p \in \mathbb P^2$, and consider the following diagram:
$$ \xymatrix{
H^3 \ar[d]_\pi & H^3 \times_{\mathbb P^2} \mathrm{Bl}_p\mathbb P^2 \ar[l] \ar[dl]|f \ar[d]^{\pi'} \\
\mathbb P^2 & \mathrm{Bl}_p\mathbb P^2 \ar[l]^{\sigma}
} $$
where in fact $H^3 \times_{\mathbb P^2} \mathrm{Bl}_p\mathbb P^2 \simeq\mathrm{Bl}_{\pi^{-1}(p)} H^3$.
We consider the fibration
$$ f:=\pi'\circ\sigma \colon H^3 \times_{\mathbb P^2} \mathrm{Bl}_p\mathbb P^2 \to \mathbb P^2 . $$
Over any point $q \neq p$, the fibre is the same as the fibre of $\pi$ at $p$, namely, the smooth elliptic curve $\mathcal E := \mathbb C \setminus\{0\} / \mathbb Z$. Over $p$, the fibre is $f^{-1}(p)=\mathbb P^1 \times \mathcal E$. Therefore the family of elliptic curves degenerates to a K\"ahler surface.

In this same example, we also notice that $H^3 \times_{\mathbb P^2} \mathrm{Bl}_p\mathbb P^2$ is a non-K\"ahler manifold fibring over $\mathrm{Bl}_p\mathbb P^2$ via $\pi'$, with a K\"ahler fibre. In fact, $H^3 \times_{\mathbb P^2} \mathrm{Bl}_p\mathbb P^2$ does not admit any lcK metric. Indeed, let $(\omega,\theta)$ be an lcK metric on $H^3 \times_{\mathbb P^2} \mathrm{Bl}_p\mathbb P^2$. Then $\omega_{\vert\mathbb P^1\times \mathcal E}$ is lcK on a K\"ahler surface, whence $[\theta]_{\vert\mathbb P^1\times\mathcal E}=0$. In particular, $[\theta]_{\vert\mathcal E}=0$. Therefore $[\theta]=(\pi')^{-1}([\alpha])$ for $[\alpha]\in H^1(\mathrm{Bl}_p\mathbb P^2;\mathbb R)$. By Lemma \ref{lem:fibration} in the next section, then $[\theta]=0$.
\end{exa}

\begin{exa}\label{ex2}
Consider again the blow-up of $\mathbb P^2$ at a point, $\mathrm{Bl}_p\mathbb P^2$. Since $\mathrm{Bl}_p\mathbb P^2$ is  a projective manifold, let us embed it into $\mathbb P^N$ for some $N$ ($N=5$ works). Consider the Hopf manifold of dimension $N+1$ and the diagram
$$ \xymatrix{
H^{N+1} \ar[d]_\pi & \pi^{-1}(\mathrm{Bl}_p\mathbb P^2) \ar[d]^\pi \ar[l] \ar[dr]^{f} & \\
\mathbb P^N & \mathrm{Bl}_p\mathbb P^2 \ar@{_(->}[l] \ar[r] & \mathbb P^2 
} $$
Here the restriction of $\pi$ to any curve $C\subset \mathrm{Bl}_p\mathbb P^2$ is a non-trivial elliptic principal bundle, therefore $f^{-1}(p)$ is a Hopf surface, in particular, non-K\"ahler.
\end{exa}

\subsection*{Multiple fibres}
The next example is meant to illustrate that usually one also has multiple fibres, but they can be ``resolved'' by means of finite Galois coverings of the base.

\begin{exa}\label{ex3} Let ${a, b, c}$ be  positive real numbers for which there exist $n_1, n_2, n_3\in \mathbb N\setminus\{0\}$  such that $a^{n_1}=b^{n_2}=c^{n_3}>1$. Define $H_{a,b,c}:=\mathbb C^3\setminus \{0\} / \mathbb Z$ where the generator $g$ of $\mathbb Z$ acts by
$$g\cdot (z_1, z_2, z_3):=(az_1, bz_2, cz_3).$$
Then one has a map $f:H_{a, b,c }\ra \P^2$ defined by
$$f(z_1, z_2, z_3):=[z_1^{n_1} : z_2^{n_2} : z_3^{n_3}].$$

Then $f$ has multiple fibres: namely, for $h\neq0$ and $k\neq0$, the fibre $f^{-1}([0:h:k])$ has multiplicity $n_1$, the fibre $f^{-1}([h:0:k])$ has multiplicity $n_2$, the fibre $f^{-1}([h:k:0])$ has multiplicity $n_3$; the fibre $f^{-1}([0:0:1])$ has multiplicity $n_1\cdot n_2$, the fibre $f^{-1}([0:1:0])$ has multiplicity $n_1\cdot n_3$, the fibre $f^{-1}([1:0:0])$ has multiplicity $n_2\cdot n_3$.

Let $\mu$ be the common value
$\mu:=a^{n_1}=b^{n_2}=c^{n_3}$, and let $\lambda$ be any solution of $\lambda^{n_1n_2n_3}=\mu$ such that $\lambda^{n_2n_3}=a$, $\lambda^{n_1n_3}=b$, $\lambda^{n_1n_2}=c$. Then we have a natural map 
$\pi\colon H_{\lambda, \lambda, \lambda}\ra H_{a, b,c }$
$$\pi(z_1, z_2, z_3)=(z_1^{n_2n_3}, z_2^{n_1n_3}, z_3^{n_1n_2})$$
whch is a (finite) Galois covering. 
In fact, the natural map
$$ f' \colon H_{\lambda, \lambda, \lambda}\to\mathbb P^2, \qquad f'(z_1, z_2, z_3):=[z_1: z_2: z_3]$$
turns  $H_{\lambda, \lambda, \lambda}$ into an elliptic principal bundle, and $H_{\lambda, \lambda, \lambda}$ is the fibered product obtained from the diagram
$$ \xymatrix{
H_{a,b,c} \ar[d]_f & H_{a,b,c} \times_{\mathbb P^2}  \mathbb P^2 =H_{\lambda, \lambda, \lambda}\ar[l] \ar[d]^{f'} \\
\mathbb P^2 & \mathbb P^2 \ar[l]^{\pi'}
} $$
where the map $\pi'$ is given by
$$\pi'([z_1: z_2: z_3])=[z_1^{n_1n_2n_3}: z_2^{n_1n_2n_3}: z_3^{n_1n_2n_3}].$$

Put in other words, $H_{\lambda, \lambda, \lambda}$ is a finite Galois cover of $H_{a, b,c}$ (with deck group $G=\sfrac{\Z}{n_1\Z}\times \sfrac{\Z}{n_2\Z} \times \sfrac{\Z}{n_3\Z}$). That is, one can pass form an elliptic fibration with multiple fibres to one with no multiple fibres by means of a suitable finite cover.
\end{exa}

We recall next some basic facts about the {\em $n$'th root fibration} or {\em cyclic Galois cover}, a procedure  that replaces non-reduced fibres of a morphism by reduced ones. The next Lemma is taken almost verbatim form \cite[Lemma 8.3, Propositions 9.1 and 9.2 at pages 111--114]{bhpv}; we include a sketch of proof here for the sake of completeness.

\begin{lem}\label{lemnthroot}
Suppose $X$ is a complex manifold, $B\subset \C^N$ is a polydisc centred at origin $0\in \C^N,$  $f\colon X\to B$ a (proper, connected) surjective holomorphic map.
Letting $D\subset B$ be $D:=\{(z_1,\dots, z_N) \;:\; z_1=0\}$ and  letting $X_0:=f^{-1}(D)$ (with scheme structure), suppose there exists some Cartier divisor $F$ on $X$ with the property 
${\mathcal O}_X(X_0)\simeq {\mathcal O}_X(nF)$ for some $n\in \N^*$.

Then, there exist a manifold $X''$ and holomorphic surjective maps $\varphi \colon X''\to X$, $\psi \colon B\to B$ and  $f''\colon X''\to B$ such that:
 $\varphi$ is \'etale,
 $\psi$ is ramified at along $D$,
 $f''$ is (proper, connected) surjective holomorphic map and makes the following diagram commutative:
$$
\xymatrix{
X'' \ar[r]^{\varphi} \ar[d]_{f''} & X \ar[d]^{f} \\
B \ar[r]_{\psi} & B
}
$$
Moreover, for every point $x$ in (the scheme-theoretical) fibre of $f''$ over $D$, one has 
$$ \mathrm{mult}_{x}({f''}^{-1}(D)) = \frac{1}{n} \mathrm{mult}_{\varphi(x)}(X_0). $$
In particular, if $F$ is smooth, then ${f''}^{-1}(D)$ is smooth, that is $f''$ is a submersion in some neighbourhood of ${f''}^{-1}(D)$.
\end{lem}

\begin{proof}[Sketch of the proof.]
Consider the map $\psi \colon B\to B$, $\psi(z_1,z_2, \dots, z_N):=(z_1^n, z_2,\dots, z_N).$
Write $f$ as
$f(x)=(f_1(x),f_2(x),\dots, f_N(x));$ then the  assumption ${\mathcal O}_X(X_0)\simeq {\mathcal O}_X(nF)$ tells us that one can find an open cover  ${\mathcal U}=\{U_\alpha\}_{\alpha \in I}$ of $X$ such that on each $U_\alpha$ there exists some holomorphic function $g_\alpha$ with the property
${f_1}_{\vert U_\alpha}=g_\alpha^n.$
Take $X'$ to be the fibred product
$$ 
\xymatrix{
X \ar[d]_{f} & X':= X\times_B B \ar[d]^{\mathrm{pr}_2} \ar[l]_{\mathrm{pr}_1\qquad} \\
B & B \ar[l]^{\psi}
}
$$
In a neighbourhood of $D$, $X'$ can be covered by open subsets $S_\alpha$ such that in each such $S_\alpha,$  $X'$ is given by  
\begin{eqnarray}\label{multip}
\left\{
\begin{array}{lcc}
z_1^n=g_\alpha^n(x), \\
z_i=f_i(x), \quad \text{ for } i=2,\dots, N.
\end{array}
\right.
\end{eqnarray}
Hence $X'$ is  a reducible space: each $S_\alpha$ above is the union of sheets  $S_{\alpha, j}$, for $j=1,\dots, n$, given by
$$S_{\alpha, j}: \quad z_1=\varepsilon^j g_\alpha, \; z_i=f_i(x), \text{ for } i=2,\dots, N,$$
where $\varepsilon$ is a fixed primitive $n$-root of unity. Notice also that the intersection $\bigcap_{j=1}^n S_{\alpha, j}$ is precisely the inverse image of $X_0$ under $\mathrm{pr}_1$ and this is the singular locus of $S_\alpha.$
Remark that each sheet $S_{\alpha, j}$ is in fact a {\em manifold}, regardless of the properties of the functions $g_\alpha.$
Next, we proceed as for the ``Riemann surface of the radical'': we disjoin  the $S_{\alpha, j}$  (we blow-up the singular locus of $S_\alpha$) and then glue  sheets $S_{\alpha, j}$ to $S_{\beta, k}$  according to the analytic continuation. In this way we get a {\em manifold} $X''$, as the sheets are smooth; the map $\mathrm{pr}_1$ induces a holomorphic and {\em \'etale} map $\varphi$ (since it is, in fact, locally  the projection of a graph), the map $\mathrm{pr}_2$ induces the map $f''$ and the inverse image of $X_0$ is now given locally by  $g_\alpha=0$, which proves the commutation of the diagram in the statement. Put it differently, the (scheme-theoretical) 
fibre of $f''$ over $D$ is now isomorphic to $F.$ Finally, the assertion on multiplicities follows from \eqref{multip}.
\end{proof}

\subsection*{Bimeromorphisms}
We recall the basic facts on factorization of bimeromorphic maps, according to \cite{akmw}.
The {\em Weak Factorization Theorem} \cite[Theorem 0.3.1]{akmw} assures that any bimeromorphism $\psi$ can be factorized as a sequence of blow-ups and blow-downs with smooth centres.
It is conjectured that the following stronger result holds true, and we will further on  assume it through the paper:
\begin{conj}[Strong Factorization Conjecture \cite{hironaka, akmw}]\label{ass:strong-fact}
Any  bimeromophism $\xymatrix{
\psi \colon X \ar@{<-->}[r] & X^* }$  between compact complex manifolds can be factorized as
\begin{equation}\label{eq:alg-red-assumption}
\xymatrix{
& \hat X \ar[ld]_{\sigma} \ar[rd]^{c} & \\
X \ar@{<-->}[rr]_{\psi} & & X^* 
}
\end{equation}
where $\sigma$ and $c$ are (compositions of) blow-ups with smooth centres.
\end{conj}

\subsection*{Non-K\"ahler surfaces}
We recollect below the basic facts about  non-K\"ahler sufaces with algebraic dimension one; the main references here are \cite{kodaira-ann1, kodaira-ann23, ueno, bhpv, brinzanescu}.

If $X$ is a (smooth compact complex) surface $X$ with $a(X)=1$, then there exists a projective curve $B$ and a surjective holomorphic map $f \colon X \to B$ whose general fibres are smooth elliptic curves \cite[Theorem 12.4]{ueno}.

Next, if $X$ is non-K\"ahler, then all the smooth fibres of $f$ are isomorphic. If $X$ is also a {\em minimal model}, that is, $K_X$ is nef, then the structure of the singular fibres of $f$ is also easy to determine: any such fibre is actually a multiple fibre, whose reduction is isomorphic to the general fibre \cite[Proposition 3.17]{brinzanescu}. It follows that, for a general such $X$ ({\itshape i.e.} not necessarily minimal), the singular fibres are all of the form
$$m \mathcal E+\text{tree of rational curves}$$
where $\mathcal E$ is the general fibre and $m\geq 1.$

\subsection*{Locally conformally K\"ahler structures}
In higher dimension, we work in the non-K\"ahler setting, more precisely, in the realm of locally conformally K\"ahler geometry. We recall that a {\em locally conformally K\"ahler structure} (lcK, for short) on a complex manifold is a Hermitian metric that locally admits a conformal change to a K\"ahler metric. If $\omega$ denotes the associated $(1,1)$-form of the metric, then the lcK condition corresponds to the equations
$$ d\omega=\theta\wedge\omega, \qquad d\theta=0 ,$$
where $\theta$ is a $1$-form, called the {\em Lee form} of $\omega$.
More precisely, if $\theta\stackrel{\text{loc}}{=}df$ by the Poincar\'e Lemma, then $\exp(-f)\omega$ is a local K\"ahler metric.

On the connected component of two such open sets where $\omega$ is conformally K\"ahler, the corresponding K\"ahler metrics differ up to homothety: in this sense, locally conformally K\"ahler geometry can be intended as a sort of ``equivariant homothetic K\"ahler geometry'' \cite{gini-ornea-parton-piccinni}.

We refer to \cite{dragomir-ornea, ornea, bazzoni-EMSS} for an open-ended account of lcK geometry and its applications.
We are interested in lcK geometry as a fruitful extension of the K\"ahler condition. In this direction, it is promising that almost any compact complex surface admits an lcK metric, with the exception of some Inoue surfaces \cite{belgun}, and perhaps of hypotethical non-Kato surfaces in class VII.

One immediately sees that the class $[\theta]\in H^1(X;\mathbb R)$ being zero corresponds to $\omega$ being globally conformally K\"ahler. On the other hand, it has been proven by Vaisman \cite{vaisman} that an lcK structure on a compact K\"ahler manifold (or, more generally, a compact complex manifold satisfying the $\partial\overline\partial$-Lemma \cite[Proposition 5.1]{kotschick-kokarev}) is necessarily gcK.
{\itshape Hereafter, when talking about lcK structures, we always assume the Lee class to be non-zero.}

Coming back to the diagram \eqref{eq:alg-red-assumption}, we notice that, if $X$ is (non-K\"ahler) lcK, then also $X^*$ is non-K\"ahler. Otherwise, $\hat X$ would be K\"ahler too, see {\itshape e.g.} \cite[Proposition 3.24]{voisin-1}. Then, by Lemma \ref{lem:vaisman}, we would get $[\theta]=0$.

\section{Results}

\subsection{\textbf{Weak locally conformally K\"ahler structures}}\label{sec:wlck}
It is known that the property of admitting a K\"ahler metric is not stable under bimeromorphisms: more precisely, it is preserved under blow-up, see {\itshape e.g.} \cite[Proposition 3.24]{voisin-1}, but in general it is not preserved under blow-down \cite{hironaka-Ann}, except when the centre is a point \cite{miyaoka}. On the other side, the lcK property is not preserved even under blow-up, in general (see Example \ref{ex1}), except when the centre is a point \cite{tricerri, vuletescu}.
For this reason, we introduce the following weaker notion:

\begin{defi}\label{defi:wlck}
A {\em weak locally conformally K\"ahler} (shortly, wlcK) structure on a complex manifold $X$ is given by a $(1,1)$-form $\omega$ and a real $1$-form $\theta$ such that
\begin{itemize}
\item $d\omega=\theta\wedge\omega$ and $d\theta=0$, and
\item $\omega>0$ outside a proper analytic subset, that we will call the {\em bad locus} $N_\omega$.
\end{itemize}
\end{defi}

This notion is motivated by the following straightforward remark:
\begin{prop}\label{prop:wlck-blowup}
Let $\sigma \colon \hat X \to X$ be a proper modification ({\itshape e.g.} a composition of blow-ups). If $X$ admits a wlcK metric, then $\hat X$ does, too.
\end{prop}

\begin{rmk}
The manifold $X$ underlying $\mathbb S^3\times\mathbb S^3$ with the Calabi-Eckmann complex structure does not admit any wlcK structure. Indeed, if $(\omega, \theta)$ would be a wlcK structure, then since $b_1(X)=0$, we get $\theta$ exact. That is, by a conformal rescaling, we can assume $\omega$ satisfying $d\omega=0$. As $b_2(X)=0$, we get $\omega $ is exact, hence $\int_X\omega^3=0$. On the other side, since $\omega>0$ outside a set of Lebesgue measure zero, we get $\int_X\omega^3>0$: contradiction.  
\end{rmk}

In the next sections, we will use the following lemmata. The first one deals with wlcK structures on fibrations:

\begin{lem}[``Lemma on fibrations'']\label{lem:fibration}
Let $X$ and $B$ be complex manifolds with $\dim X > \dim B$.
Let $f\colon X \to B$ be a surjective proper holomorphic map with connected fibres.
Let $(\omega, \theta)$ be a wlcK structure on $X$.  If the Lee class $[\theta]=f^*[\alpha]$ in the image of the pull-back induced by $f$, where $[\alpha]\in H^1(B;\mathbb R)$, then $[\theta]=0$.
\end{lem}

\begin{proof}
Up to global conformal changes, let us assume $\theta=f^*\alpha$. We denote by $S_f$ the union of the non-flat and singular loci, namely, the set of points $b\in B$ such that the corresponding fibre $f^{-1}(b)$ is either singular, or has dimension strictly greater than $k:=\dim X - \dim B>0$. That $S_f$ is an analytic subspace. Moreover, the restriction of $f$ to $X \setminus f^{-1}(S_f) \to B \setminus S_f$ is now a submersion. We also consider the open subset $B' \subset B \setminus S_f$ such that, for any $b\in B'$, we have
$$ v(b) := (f_* \omega^k)(b) = \int_{f^{-1}(b)} \omega^k_{\vert f^{-1}(b)} = \mathrm{vol}_\omega(f^{-1}(b)) >0 $$
is a positive function. By the ``weak'' assumption on the metric, the complement of $B'$ in $B$ is a proper analytic subset.
By the wlcK condition $d\omega=f^*\alpha\wedge\omega$ and by the projection formula, we compute
$$ d(f_*\omega^k) = f_*(d\omega^k) = k f_* (f^*\alpha\wedge\omega^k) = k \alpha \wedge f_*\omega^k , $$
that we rewrite as
$$ d v = k \alpha \wedge v. $$
This says that
$$ \alpha_{\vert B'} = \frac{1}{k} \frac{dv}{v} = \frac{1}{k} d \lg v , $$
namely, $[\alpha]_{\vert B'}=0$.
The map $H^1(B;\mathbb R) \to H^1(B';\mathbb R)$ induced by the inclusion is injective. This follows, for example, by taking the long exact sequence in cohomology for the pair $(B,B')$ and by noticing that $H^1(B,B';\mathbb R)=0$, by excision, homotopy invariance, and the Thom isomorphism, and by degree reason since the real codimension of $B\setminus B'$ in $B$ is greater or equal than $2$, see {\itshape e.g.} \cite[Theorem II.5.3]{suwa}. Therefore we get that $[\alpha]=0$ in $H^1(B;\mathbb R)$, whence the statement.
\end{proof}

The second lemma is a generalization of the Vaisman lemma \cite{vaisman} to the weak lcK context:

\begin{lem}\label{lem:vaisman}
Let $X$ be a compact complex manifold of $\dim X>1$ endowed with a wlcK structure $(\omega, \theta)$. If $X$ admits K\"ahler metrics, then $[\theta]=0$.
\end{lem}

\begin{proof}
Denote by $n$ the complex dimension of $X$.
Fix $g_0$ a K\"ahler metric on $X$. Let $\theta_0$ be the harmonic representative of $[\theta]$ with respect to $g_0$. Up to global conformal change, we can take the wlcK structure $\omega_0$ with Lee form $\theta_0$. Let $\theta_0=\alpha+\bar\alpha$ be the decomposition of $\theta_0$ into pure-type components, where $\alpha\in \wedge^{1,0}X$. Multiplying the condition $d\omega_0^{n-1}=(n-1)\theta_0\wedge\omega_0^{n-1}$ by $\alpha$, we get
$$ \alpha\wedge d\omega_0^{n-1} = (n-1) \alpha\wedge\bar\alpha\wedge\omega_0^{n-1} .$$
By the strong Hodge decomposition, $\alpha$ is itself harmonic: in particular, $d\alpha=0$. 
It follows that
$$ 0 = \int_X \alpha \wedge d\omega_0^{n-1} = (n-1) \|\alpha\|^2_{{\omega_0}_{\vert X\setminus N_\omega}} . $$
Therefore $\alpha=0$, yielding the statement.
\end{proof}

\begin{lem}\label{lem:ansp2}
Let $X$ be an lcK manifold, $Z\subset X$ a compact complex subspace (with reduced structure) of $\dim Z>1$. If its desingularization $\hat{Z}$ is of K\"ahler type, then $[\theta]_{\vert Z}=0.$
\end{lem}
\begin{proof} Blowing-up $X$ along the singularities of $Z$, we get a wlcK manifold $\hat{X}$ that contains $\hat{Z}.$ Since $sing(Z)\not=Z$  (as $Z$ is reduced) we get that $\hat{Z}$ inherits a wlcK structure, so we apply the previous Lemma.
\end{proof}

\subsection{\textbf{General fibres of the algebraic reduction}}
Under the above assumptions, we can prove the first result: as in the case of surfaces, the general fibre of the algebraic reduction has constant $j$-invariant. The reasoning is pretty much an extension of Example \ref{ex1}, using this time the ``Lemma on fibrations" \ref{lem:fibration} in its more general setup of wlcK. 

\begin{thm}\label{thm:smooth-fibres}
Let $X$ be a compact complex threefold with $a(X)=2$, endowed with an lcK structure $(\omega,\theta)$. Assume the Strong Factorization Conjecture, and consider the algebraic reduction
\begin{equation}\label{eq:alg-red-thm}
\xymatrix{
& \hat X \ar[ld]_{\sigma} \ar[rd]^{c} \ar@/^1.2cm/[ddr]|F & \\
X \ar@{<-->}[rr]_{\psi} & & X^* \ar[d]	_{f} \\
& & B .
}
\end{equation}
Then all the smooth one-dimensional fibres of $f$ are isomorphic.
\end{thm}

\begin{proof}
It suffices to prove that the general fibres of $f$ are isomorphic, since by continuity of the $j$-invariant, the claim follows.

Denote by $E$ the exceptional divisor of the blow-up $\sigma$.
Consider $F=f \circ c \colon \hat X \to B$.
Fix a very ample line bundle $\mathcal H \in \mathrm{Pic}(B)$ on the projective surface $B$, and set $\mathcal L:=F^*(\mathcal H)$. Note that $\mathcal L$ is globally generated. By applying twice the Bertini theorem for a general section $H\in |\mathcal H|$, both $H$ and $S_H:=F^{-1}(H)\in |\mathcal L|$ are smooth, and we can also assume that $S_H \cap E \subsetneq S_H$. The last fact assures that $\sigma^*(\omega)_{\vert S_H}$ is wlcK.

We claim that $S_H$ is non-K\"ahler. On the contrary, if $S_H$ was K\"ahler, then $[\theta]_{\vert S_H}=0$ in $H^1(S_H;\mathbb R)$ by Lemma \ref{lem:vaisman}. In particular, the restriction of the class $[\theta]$ to the general fibre of $F$ vanishes, namely, there exists an analytic subset $\mathcal Z \subset B$ such that $F\colon \hat X \setminus F^{-1}(\mathcal Z) \to B \setminus \mathcal Z$ is a submersion with fibre ${\mathcal E}$ and such that $[\theta]_{\vert\mathcal E}=0$. By the long exact sequence of the fibration, we get $[\theta]_{\vert\hat X \setminus F^{-1}(\mathcal Z)}=F^*([\alpha])$ for some $[\alpha] \in H^1(B \setminus \mathcal Z; \mathbb R)$. By Lemma \ref{lem:fibration}, we get $[\theta]_{\vert\hat X \setminus F^{-1}(\mathcal Z)}=0$. Since the map $H^1(\hat X; \mathbb R) \to H^1(\hat X \setminus F^{-1}(\mathcal Z); \mathbb R)$ is injective and the map $H^1(\hat X;\mathbb R) \to H^1(X;\mathbb R)$ is an isomorphism, we get that $[\theta]=0$ in $H^1(X; \mathbb R)$.

Therefore, $S_H$ is non-K\"ahler. By a theorem of  \cite[Proposition 3.17]{brinzanescu}, we get that the general fibres of $S_H \to H$ are isomorphic. Therefore, the general fibres of $f$ are isomorphic.
\end{proof}

\begin{rmk}
By induction, the statement of Theorem \ref{thm:smooth-fibres} holds true for lcK manifolds $X$ of arbitrary dimension $n$ having $a(X)=n-1$.
\end{rmk}

\subsection{\textbf{Singular fibres of the algebraic reduction}}
First, we state a lemma which will be used later. To emphasize the usefulness of the lemma, notice that for elliptic (K\"ahler) surfaces there may exist isolated fibres that are singular when given the reduced structure \cite{kodaira-ann23}. The next lemma simply says that such a situation cannot occur for our $3$-folds. A naive reason for this can be found also at the topological level: a fibre with singular reduction would impose a non-trivial local monodromy around it, but since the $2$-dimensional ball with origin removed $B^\circ$ is simply-connected -- in contrast with the disc with removed origin -- such a situation cannot occur.

\begin{lem}\label{singo}
Denote by $B^\circ$ be the $2$-dimensional ball with origin removed, $B^\circ:=B^2\setminus\{0\}$, and let $U$ be a $3$-manifold such that there is a flat map $f\colon U\to B^2.$ Assume that all the fibres of $f$ are smooth isomorphic elliptic curves, call them $\mathcal E$, except maybe for the central one, $F_0:=f^{-1}(0)$. Then the central fibre $F_0$ is smooth as well, and isomorphic to the general fibre.
\end{lem}

\begin{proof}
Let $V:=U\setminus F_0:$ then the map $f_0:=f_{\vert V} \colon V\to B^\circ$ is a submersion. Since all fibres of $f_0$ are smooth isomorphic, by the Fischer-Grauert theorem and by \cite[Corollary 2.12]{brinzanescu}, we see that $f_0$ is actually an elliptic bundle ({\itshape a priori} not necessarily principal). Using the same notations as in \cite[Section V.5]{bhpv}, we denote by  ${\mathcal A}_{\mathcal E}$  group of automorphisms of $\mathcal E$: then we have a short exact sequence of (non-Abelian) groups 
$$0\to \mathcal E \to  {\mathcal A}_{\mathcal E}\to {\mathbb Z}_m\to 0$$
for some $m\in\{2,4,6\}$.
We denote by the same letters $\mathcal E$, respectively $\mathcal A_{\mathcal E}$, the sheaves of germs of holomorphic functions on $B$ with values in $\mathcal E$, respectively $\mathcal A_{\mathcal E}$. Then the set of elliptic bundles with typical fibre $\mathcal E$ over $B^\circ$ is in natural bijection with $H^1(B^\circ; {\mathcal A}_{\mathcal E})$, while the set of principal elliptic bundles is parametrised by $H^1(B^\circ; {\mathcal E}).$  The above exact sequence and the fact that $H^1(B^\circ; {\mathbb Z}_m)$ is trivial tells us that in fact the map $f_0$ is a principal bundle. 

Next, we fix a point $P$ in the regular part of  $F_0$; one can find local coordinates $(z_1, z_2, z_3)$ on $X$ in some neighbourhood of $P$ such that $P=(0,0,0)$ and  $F_0$ is given by $\{z_2=z_3=0\}$. Then the set $\Sigma:=\left\{(0, z_2,z_3) \;\;:\;\; \vert z_2\vert, \vert z_3\vert \text{ small enough} \right\}$ defines a local section of the principal bundle above. Up to take a smaller ball, it follows that $f_0$ is actually trivial, in  particular,	 there exists an isomorphism
$\varphi^\circ \colon U\setminus F_0\ra B^\circ\times \mathcal E$. We claim that $\varphi^\circ$ extends to an isomorphism  $\varphi \colon U\ra B\times \mathcal E$.

Indeed, the image  $\Sigma':=\varphi^\circ(\Sigma)$  is a section of $B^\circ\times \mathcal E$; it follows that  $B^\circ\times \mathcal E\setminus \Sigma'$ is Stein, hence the restriction of $\varphi^\circ$ to $(U\setminus \Sigma)\setminus F_0 \subset U \setminus \Sigma \subset{\mathbb C}^4$ extends over $F_0\cap (U\setminus \Sigma)$ to a holomorphic map. Hence we extended $\varphi^\circ$ to all the points of $U$ except $P$, in particular along $F_0\setminus P$: but then Hartogs shows that eventually $\varphi^\circ$ extends also over $P.$

Now clearly the extension $\varphi$ remains an isomorphism, since it is an isomorphism is codimension two.
\end{proof}

\begin{rmk}
If we assume that the central fibre of $f$ is reduced, then the Lemma follows directly from Looijenga's theorem on the purity of the discriminant locus (see \cite[Theorem 4.8 at page 63]{looijenga}), since our fibration being flat satisfies the ``isolated complete intersection singularity'' (icis) condition.
\end{rmk}

We now understand the singular fibres of the algebraic reduction. To this aim, we also assume that $X^*$ in the algebraic reduction \eqref{eq:alg-red} is a {\em minimal model} (in the sense of Mori), in other words, that $K_{X^*}$ is {\em nef}, namely, $(K_{X^*} . C)\geq0$ for any curve $C$ on $X^*$.

For the next theorem, compare Example \ref{ex3}.

\begin{thm}\label{thm:quasibundle}
Let $X$ be a compact complex threefold with $a(X)=2$, endowed with an lcK structure $(\omega,\theta)$. Assume the Strong Factorization Conjecture, and consider the algebraic reduction as in \eqref{eq:alg-red-thm}.
Assume that $f$ is flat and $S(f)$ is a simple normal crossing divisor.
Assume moreover that $K_{X^*}$ is nef.
Then all fibres of $f$ with the reduced structure are elliptic and isomorphic, that is $X^*$ is a quasi-bundle.
\end{thm}

\begin{proof}
First notice that, by the Bertini theorem, through the fibre over the general point of $S(f)$ there is a smooth $S_H$. By \cite[Corollary 3.17]{brinzanescu}, it follows that the general singular fibre of $f$ is of the form
$$m\mathcal{E}+\text{ tree of rational curves},$$
for $m\in\mathbb N$.

The tree of rational curves is easily excluded from the assumption that $K_{X^*}$ is nef.
Indeed, let $C \subset X^*$ be a rational curve in the fibre. By the adjunction formula, $K_C = (K_{X^*})_{\vert C} \otimes \det \mathcal{N}_{C\vert X^*}$, where $\mathcal{N}_{C\vert X^*}=\mathcal O_C^{\oplus2}$ since $C$ is in the fibre. Therefore $\deg K_C=\deg (K_{X^*})_{\vert C} = (K_{X^*} . C)$. On the one side, $\deg K_C=0$, the curve $C$ being rational; on the other side $(K_{X^*} . C)\geq0$ by the nef assumption. This is absurd.

Therefore the general singular fibre is of the form $m\mathcal{E}$; notice that $\mathcal{E}$ is isomorphic to the general fibre by the continuity of $j$-invariant. 
It is easy to show that the multiplicity $m$ is locally constant. If a singular fibre would not be of this form $m\mathcal{E}$, then it lives either over an smooth point of the singular locus and all the neighbouring fibres are of a given multiplicity $m$, or lives above a node of $S(f)$ and the fibres over its neighbouring  points have multiplicities $m$, respectively $n$ on each branch of $S(f)$.

The assertion is local, so we will use Lemma \ref{lemnthroot}.
 Taking a Galois cover of order $m$ along $S(f)$ in the first case, or of order $m$ along the first branch and next of order $mn$ along the proper transform of the second, we may now  assume that $f$ has finitely many singular fibres. But using the Lemma \ref{singo} above, we conclude that all the singular fibres of $f$ are just multiple structures on the smooth elliptic general fibre. 
\end{proof}

\subsection{\textbf{Structure of lcK $3$-folds}}
Finally, we can further describe the structure of the algebraic reduction:

\begin{thm}\label{thm:structure}
Let $X$ be a compact complex threefold with $a(X)=2$, endowed with an lcK structure $(\omega,\theta)$. Assume the Strong Factorization Conjecture, and consider the algebraic reduction as in \eqref{eq:alg-red-thm} .
Assume that $\hat{X}$ is lcK, and also that  $f$ is flat, $S(f)$ is a simple normal crossing divisor and  $K_{X^*}$ is nef.
Then:
\begin{itemize}
\item the map $\sigma$ in \eqref{eq:alg-red-thm} can be taken to be the identity;
\item and the blow-ups in $c$ are ``special", in the sense that any blow-up in it is either a blow-up of a point, or a blow-up along a  curve contained in a (previously produced) exceptional divisor.
\end{itemize}
\end{thm}

\begin{proof} Let us proceed with the proof.

\begin{description}
\item[Step1] {\em  We claim that any  divisor  $E\subset \hat{X}$ contracted by $\sigma$ is also contracted under $c$.}
Here, by ``contracted'' by $c$, we mean that $\dim(c(E))<\dim(E)$.

\begin{proof}[Proof of Step 1]
 On the contrary, assume that $E$ is not contracted by $c$, namely, $S:=c(E)\subset X^*$ is a surface. We look at the image $f(S)\subseteq B$.

Three cases may occur: the image of $S$ under $f$ is either the whole $B$, or a curve in $B$, or just a point.

If $f(S)=B$, then $S$ intersects the general fibre in at least one and finitely-many points. On the other side, by Nakayama \cite[Theorem 7.4.4]{nakayama-global}, the Poincar\'e dual $\mathrm{PD}(F)$ of the general fibre $F$ is $0$, whence 
$$ S . F = \int_{X^*} \mathrm{PD}(S) \wedge \mathrm{PD}(F) = 0 . $$

We consider the case when $f(S)$ is a curve and $Z$ is a point. Then $E=\mathbb{P}^2$. In this case, $f\circ c$ would be a non-constant map from $\mathbb{P}^2$ to a curve, which is not possible, because any two curves in $\mathbb P^2$ intersect.

Let us consider now the case $f(S)$ is a curve and $Z$ is a curve. In this case, $E$ is a ruled surface: let $\ell\simeq \P^1$ be a general member of its ruling, and $\ell'\subset X^*$ its image under $c$. Notice that 
\begin{eqnarray}\label{exclin}
(K_{\hat{X}}.\ell)=-1.
\end{eqnarray}
 If $E$ is not contracted by $c$, we get that $\ell'$ is a curve and $c\vert _{\ell}:\ell\ra \ell'$ is its normalization. As $K_{X^*}$ is nef, we have $(K_{X^*}.\ell')\geq 0.$ But since $c$ is a blow-up, we get that $K_{\hat{X}}=c^*(K_{X^*})+\Delta$ where $\Delta$ is some effective divisor on $\hat{X}$; notice that $\ell$ is not included in $\Delta$. But then  $(K_{\hat{X}}.\ell)=(c^*K_{X^*}.\ell)+(\Delta. \ell)\geq  (c^*K_{X^*}.\ell)= (K_{X^*}.\ell')\geq 0$, contradicting \eqref{exclin}.

The last case when $f(S)$ is a point is excluded by the assumption that $f$ is flat.
\end{proof}

\medskip

\item[Step 2] We already know that $X^*$ is a quasi-bundle: we infer that {\em the only curves on $X^*$ are the fibres of $f$.}

\begin{proof}[Proof of Step 2]
Indeed, if $C\subset X^*$ would be a transverse curve, let  $D:=f(C)$ and $T:=f^{-1}(D)$. Notice that $T$ is a fibration over $D$ whose fibres are among the fibres of $f$; then the desingularization $\hat{T}$ is also a fibration over $D$ which has a transversal curve, hence by \cite[Proposition 2.14]{brinzanescu}, we get that $\hat{T}$ is of K\"ahler type. Taking the proper transform $T'$ of $T$ to $\hat{X}$ under $c$, we get that $\hat{T}$ is wlcK; and applying Lemma \ref{lem:ansp2}, we get that $[\theta]_{\vert T}=0$, hence the restriction of $[\theta]$ to the fibres of $f$ would vanish as well. We get that $[\theta]$ would be a pull-back from $B$.
\end{proof}

\item[Step 3] At this step we prove that {\em we can remove all blow-ups in the chain $\sigma$, to get that $X$ is a blown-up of $X^*$.}

\begin{proof}[Proof of Step 3]
Write $c$ as $c=c_1\circ \cdots \circ c_m$ where $c_i \colon \hat{X}_i\to \hat{X}_{i-1}$ are blow-ups with smooth loci (with $\hat{X}_m=\hat{X}$ and $\hat X_0=X^*$) and similarly $\sigma=\sigma_1\circ \dots \sigma_n$, where $\sigma_j \colon X_j\ra X_{j-1}$ (with $X_n=\hat{X}$ and $X_0=X$). We prove that $c_m=\sigma_n$, hence $X_{n-1}=\hat{X}_{m-1}$, so we can replace $\hat{X}$ by $X_{n-1}$. By induction on $n$, we get the claim.

First we notice that, in the above chain of blowups $c_1,\dots, c_m$, {\em the first blow-up $c_1$ is necesssarily the blow-up of a point in $X^*=\hat X_0$}.
On the contrary, suppose that $c_1$ is the blow-up of a curve and let $E_1\subset \hat X_1$ be the exceptional divisor of $c_1$ and $\hat{E_1}\subset \hat{X}$ its strict transform. Since all curves in $X^*$ are the fibres $F$ of $f$ by {\itshape Step 2}, we see $E_1\simeq F\times \P^1$  and $\hat{E_1}$ is a possible blow-up of it: in particular $\hat{E_1}$ is a K\"ahler surface.
Consider $\theta$ the Lee form of the lcK metric on $X$. Since blow-ups induce isomorphisms at the level of the first cohomology group $H^1$, we see there are unique cohomology classes $[\hat\theta]\in H^1(\hat{X})$ and $[\theta^*]\in H^1(X^*)$ corresponding to $[\theta]\in H^1(X)$. But, as $\hat E_1$ is  K\"ahler, we see that $[\hat{\theta}]_{\vert \hat E_1}=0.$
We derive that $[\theta^*]_{\vert F}=0$, hence $[\theta^*]$ is a pull-back form $H^1(B)$; but, by Lemma \ref{lem:fibration}, we would get that $\hat{X}$ is K\"ahler, absurd.

Let $E\subset \hat{X}$ be the exceptional divisor of $\sigma_n$; it is either a plane $\P^2$ or a ruled surface. By {\it Step 1}, we know that $E$ is also contracted under $c$, so it must be contracted by some $c_i$, hence $E$ would dominate the exceptional locus $E_i$ of $c_i$. If $i<m$, then $E$ is the proper transform of $E_i$ under $c_{i+1}\circ\cdots\circ c_m$, hence $E$ is a blow-up at some point of $E_i$; in particular, $E$ would be non-minimal, contradiction, except possibly for the case when $E=\Sigma_1$, the first Hirzebruch surface --- which is ${\mathbb P}^2$ blown-up at one point. But this case is easily excluded by looking at the restriction of the normal bundle of $E$ to a fibre $F$ of its ruling. Namely, if we look $E$ as being the exceptional divisor of $\sigma_n$, then $\left({\mathcal N}_{E\vert\hat{X}}\right)_{\vert F}={\mathcal O}_F(-1),$ 
while if we look at $E$ as being the proper transform of some exceptional divisor $E_i$ of some $c_i$ ($i<m$), then 
$\left({\mathcal N}_{E\vert\hat{X}}\right)_{\vert F}={\mathcal O}_F(-d)$ 
with $d\geq 2.$
Hence {\em $\sigma_n$ and $c_m$ have the same exceptional locus}.

If $E$ is the projective plane, then this already implies that $c_m=\sigma_n$, amd we are done. So the only possibility for $\sigma_n\not=c_m$ would be that $E$ is a ruled surface  that is contracted in two different ways to some smooth  curves $C_1\subset X_{n-1}$ and $C_2\subset \hat{X}_{m-1}$.

%{\color{magenta}{
So $C_2$ would be a flopping curve in $\hat{X}_{m-1}$. Since blowing-up the flopping curve $C_2$ yields a ruled surface with two rulings (because the exceptional divisor contracts to the smooth curve $C_1$), this means that this surface is $E=\P^1\times \P^1$, hence the curve is smooth rational and  the normal bundle must be of the form ${\mathcal O}(a)\oplus {\mathcal O}(a)$ for some $a\in \Z.$ Using the following lemma, we see this cannot happen, hence the theorem is proved.

\begin{lem}\label{lem:no-flop}
Let $c':= c_1\circ \dots\circ c_{m-1} \colon \hat{X}_{m-1}\to X^*$ be a composition of blow-ups with smooth centres, and assume the exceptional divisor $\Delta\subset \hat{X}_{m-1}$ of $c'$ is contracted to a point in $X^*$. Then $\Delta$ contains no flopping curves.
\end{lem}

\begin{proof}[Proof of Lemma \ref{lem:no-flop}]
By  \cite[page 1398, case (1)]{donovan-wemyss}, the  normal bundle $\mathcal{N}_{C_2\vert \hat{X}_{m-1}}$ of a flopping curve is of one of the following three types:
\begin{enumerate}
\item ${\mathcal O}_{C_2}(-1)\oplus {\mathcal O}_{C_2}(-1)$, the ``Atiyah flop'';
\item ${\mathcal O}_{C_2}\oplus {\mathcal O}_{C_2}(-2)$;
\item ${\mathcal O}_{C_2}(1)\oplus {\mathcal O}_{C_2}(-3)$.
\end{enumerate}
By the above remark on the normal bundle of $C_2$, we are left to exclude the case 1.

Let $E\subset \Delta$ be an irreducbile component of $\Delta$ containing $C_2$. The normal bundle sequence
$$
0\to 
\mathcal{N}_{{C_2}\vert E}
\to
\mathcal{N}_{{C_2}\vert \hat{X}_{m-1}}
\to
\left(\mathcal{N}_{E\vert \hat{X}_{m-1}}\right) _{\vert {C_2}}
\to
0
$$
reads in this situation
\begin{eqnarray}\label{norm}
0\to
\mathcal{N}_{C_2\vert E}
\to
{\mathcal O}_{C_2}(-1)\oplus{\mathcal O}_{C_2}(-1)\to
\left(\mathcal{N}_{E\vert \hat{X}_{m-1}}\right) _{\vert C_2}
\to
0.
\end{eqnarray}
By \cite[Theorem 3.7, combined with Theorem 2.9]{ancona}, the normal bundle of $\Delta$ in $\hat{X}_{m-1}$, namely ${\mathcal O}_{\hat{X}_{m-1}}(\Delta)$, is negative (equivalently, anti-ample) on $\Delta$. Writing $\Delta=E+R$, we see
\begin{eqnarray*}
\left(\mathcal{N}_{E\vert \hat{X}_{m-1}}\right) _{\vert C_2}
&=& {\left({\mathcal O}_{\hat{X}_{m-1}}(E)_{\vert E}\right)}_{\vert C_2} \\
&=& {\left({\mathcal O}_{\hat{X}_{m-1}}(\Delta-R)_{\vert E}\right)}_{\vert C_2} \\
&=& {\left({\mathcal O}_{\hat{X}_{m-1}}(\Delta)_{\vert \Delta}\right)}_{\vert C_2}\otimes {\left({\mathcal O}_{\hat{X}_{m-1}}(-R)_{\vert E}\right)}_{\vert C_2} .
\end{eqnarray*}
Let 
\begin{eqnarray}\label{left}
a:=-(C_2)^2=\deg \mathcal{N}_{C_2\vert E} ;
\end{eqnarray}
then $a\geq1$, since $\mathcal{N}_{C_2\vert E}$ is a subsheaf of ${\mathcal O}(-1)^{\oplus 2}.$ 
On the other hand, by the anti-ampleness of $\Delta_{\vert \Delta}$, we see
\begin{equation}\label{right}
{\left({\mathcal O}_{\hat{X}_{m-1}}(\Delta)_{\vert \Delta}\right)}_{\vert C_2} ={\mathcal O}_{C_2}(-b)\text{ for some }b\geq 1.
\end{equation}
Next, 
\begin{equation}\label{right2}
\deg{\left({\mathcal O}_{\hat{X}_{m-1}}(-R)_{\vert E}\right)}_{\vert C_2}=(-(R\cap E)\cdot C_2)=\epsilon a-k
\end{equation}
where $\epsilon=1$ or $\epsilon=0$ --- according to the fact that $C_2$ is a subset of $R\cap E$  (that is, if $C_2$ lives on more than one component of $\Delta$) or not, respectively --- and $k\geq 0$ is the intersection of $C_2$ with the components of $R\cap E$ different form $C_2$.

Put all together, the normal sequence \eqref{norm} becomes
\begin{eqnarray}
\lefteqn{ 0\to
{\mathcal O}_{C_2}(-a)
\to
{\mathcal O}_{C_2}(-1)\oplus{\mathcal O}_{C_2}(-1) }\\
&&\to
{\mathcal O}_{C_2}(-b+\epsilon a-k)
\to
0.\nonumber
\end{eqnarray}
Then, taking determinants, we get:
\begin{itemize}
\item $-2=-b-k$ when $\epsilon=1$,
\item $-2=-a-b-k$ when $\epsilon=0$.
\end{itemize} 

The first case can be eliminated as follows. We have two subcases: either $k=0$ and $b=2$, or $k=1$ and $b=1$.

If $k=0$, this means that the curve $C_2$ is the intersection of  two  components (the $E$ above and some other, say $E_1$) of $\Delta$ and meets no other component of it. Then, one of the $E$ and $E_1$ --- say $E_1$ --- is the blowup of $C_2$ as a curve living on $E$. Then, as  
$${\mathcal N}_{C_2\vert E}\oplus {\mathcal N}_{C_2\vert E_1}={\mathcal N}_{C_2\vert \hat{X}_{m-1}}={\mathcal O}_{C_2}(-1)^{\oplus 2}$$
we see 
${\mathcal N}_{C_2\vert E}\simeq {\mathcal O}_{C_2}(-1)$
hence $C_2$ would be a $(-1)$-curve,
which is a contradiction with the fact that blowing up smooth curves on threefolds produce minimal surfaces as exceptional divisors, see \cite[page 346, ``Blow up of a smooth curve'']{pinkham}.

In the subcase $k=1$ and $b=1$, the above reasoning does not apply, since the curve $C_2$ is a strict transform of a curve $C'$ of an exceptional divisor of some previous blow-up, call it $c''$. But then the normal bundle of $C_2$ in $\hat{X}_{m-1}$ would be strictly smaller than the degree of the normal bundle of $C'$, which is  negative by \cite{ancona}; so we would get $b\geq 2$ in this case, which is a contradiction again.

In the second case, as $a\geq 1$ (by \eqref{left}) and $b\geq 1$ (by \eqref{right}), we get that $a=b=1$ and $k=0$; this means that $C_2$ does not meet any other component of $\Delta$, and henceforth $C_2$ would live as an exceptional curve on a divisor produced by some blow-up, which is a contradiction with \cite{pinkham} again.
\end{proof}
%}}
\end{proof}
\end{description}
This completes the proof of the statement.
\end{proof}

%\bibliographystyle{alpha}
%\bibliography{biblio}

\end{document}